\documentclass[a4paper,11pt,english]{smfart}

\providecommand{\og}{``}
\providecommand{\fg}{''}
\providecommand{\smfandname}{\&}

\providecommand{\smfedname}{\'ed.}

\usepackage{amsmath, amsfonts, amssymb, amsthm}
\usepackage[all]{xy}
\usepackage[height=22.5cm, width=15.5cm]{geometry}

\numberwithin{equation}{section}

\theoremstyle{plain}
\newtheorem{thm}{Theorem}[section]
\newtheorem{prop}[thm]{Proposition}
\newtheorem{lem}[thm]{Lemma}
\newtheorem{cor}[thm]{Corollary}
\newtheorem{thm*}{Theorem}
\newtheorem{prop*}[thm*]{Proposition}

\theoremstyle{definition}
\newtheorem{defn}[thm]{Definition}

\theoremstyle{remark}
\newtheorem*{rem}{Remark}
\newtheorem*{ex}{Example}

\newcommand{\mbb}[1]{\mathbb{#1}}
\newcommand{\wt}[1]{\widetilde{#1}}

\newcommand{\ol}[1]{\overline{#1}}
\newcommand{\lie}[1]{{\mathfrak{#1}}}

\newcommand{\norm}[1]{\lVert #1\rVert}

\newcommand{\im}{{\mathrm{i}}}

\DeclareMathOperator{\Ad}{Ad}

\DeclareMathOperator{\Hom}{Hom}

\DeclareMathOperator{\Lie}{Lie}

\DeclareMathOperator{\rk}{rk}

\DeclareMathOperator{\gen}{gen}

\title{Spherical gradient manifolds}

\thanks{The authors would like to thank Peter Heinzner for many useful
discussions. The first author thanks the Fakult\"at f\"ur Mathematik of the
Ruhr-Universit\"at Bochum for its hospitality.}

\author{Christian Miebach}
\address{Centre de Math\'ematiques et Informatique, UMR-CNRS 6632 (LATP), 39,
rue Joliot-Curie, Universit\'e de Provence, 13453 Marseille Cedex 13 France}
\email{miebach@cmi.univ-mrs.fr}

\author{Henrik St\"otzel}
\address{Fakult\"at f\"ur Mathematik, Ruhr-Universit\"at Bochum,
Universit\"atsstra{\ss}e 150, D - 44780 Bochum}
\email{henrik.stoetzel@ruhr-uni-bochum.de}


\begin{document}

\begin{abstract}
We study the action of a real-reductive group $G=K\exp(\lie{p})$ on
real-analytic submanifold $X$ of a K\"ahler manifold $Z$. We suppose that the
action of $G$ extends holomorphically to an action of the complexified group
$G^\mbb{C}$ such that the action of a maximal Hamiltonian subgroup is
Hamiltonian. The moment map $\mu$ induces a gradient map $\mu_\lie{p}\colon
X\to\lie{p}$. We show that $\mu_\lie{p}$ almost separates the $K$--orbits if
and only if a minimal parabolic subgroup of $G$ has an open orbit. This
generalizes Brion's characterization of spherical K\"ahler manifolds with
moment maps.

\noindent
{\bf{\emph{R\'esum\'e.}}} ---\hspace{0.5em} Nous \'etudions l'action d'un
groupe r\'eel-r\'eductif $G=K\exp(\lie{p})$ sur une sous-vari\'et\'e
r\'eel-analytique $X$ d'une vari\'et\'e k\"ahl\'erienne $Z$. Nous supposons que
l'action de $G$ peut \^etre prolong\'ee \`a une action holomorphe du groupe
complexifi\'e $G^\mbb{C}$ telle que l'action d'un sous-groupe maximal compact
de $G^\mbb{C}$ soit hamiltonienne. L'application moment $\mu$ induit une
application gradient $\mu_\lie{p}\colon X\to\lie{p}$. Nous montrons que
$\mu_\lie{p}$ separe les orbites de $K$ si et seulement si un sous-groupe
minimal parabolique de $G$ poss\`ede une orbite ouverte dans $X$. Ce r\'esultat
g\'en\'eralise la charact\'erisation de Brion des vari\'et\'es k\"ahl\'eriennes
sph\'eriques qui admettent une application moment.
\end{abstract}

\maketitle

\section{Introduction}

Let $U^\mbb{C}$ be a complex-reductive Lie group with compact real form $U$ and
let $Z$ be a K\"ahler manifold on which $U^\mbb{C}$ acts holomorphically such
that $U$ acts by K\"ahler isometries. Assume furthermore that the
$U$--action on $Z$ is Hamiltonian, i.\,e.\ that there exists a $U$--equivariant
moment map $\mu\colon Z\to\lie{u}^*$ where $\lie{u}$ denotes the Lie algebra of
$U$.

In the special case that $Z$ is compact it is shown in~\cite{Bri} (see
also~\cite{HuWu}) that $\mu$ separates the $U$--orbits if and only if $Z$ is a
spherical $U^\mbb{C}$--manifold, which means that a Borel subgroup of
$U^\mbb{C}$ has an open orbit in $Z$. Note that $\mu$ separates the $U$--orbits
if and only if it induces an injective map $Z/U\hookrightarrow\lie{u}/U$.
Moreover, this is equivalent to the property that the $U$--action on $Z$ is
coisotropic.

In this paper we generalize Brion's result to actions of real-reductive groups
on real-analytic manifolds which moreover are not assumed to be compact. More
precisely, we consider a closed subgroup $G$ of $U^\mbb{C}$ which is compatible
with the Cartan decomposition $U^\mbb{C}=U\exp(\im\lie{u})$. This means that
$G=K\exp(\lie{p})$ where $K:=G\cap U$ and $\lie{p}$ is an $\Ad(K)$--invariant
subspace of $\im\lie{u}$. Let $X$ be a $G$--invariant real-analytic submanifold
of $Z$. By restriction, the moment map $\mu$ induces a $K$--equivariant
gradient map $\mu_\lie{p}\colon X\to(\im\lie{p})^*$.

There are two main differences between the complex and the real situation:
Even if $X$ is connected an open $G$--orbit in $X$ does not have to be dense
and in general the fibers of $\mu_\lie{p}$ are not connected. Therefore one
cannot expect $\mu_\lie{p}$ to separate the $K$--orbits globally in $X$. We say
that $\mu_\lie{p}$ locally almost separates the $K$--orbits if there exists a
$K$--invariant open subset $\Omega$ of $X$ such that $K\cdot x$ is open in
$\mu_\lie{p}^{-1}\bigl(K\cdot\mu_\lie{p}(x)\bigr)$ for all $x\in\Omega$.
Geometrically this means that the induced map $\Omega/K\to\lie{p}/K$ has
discrete fibers. If $\Omega=X$, we say that $\mu_\lie{p}$ almost separates the
$K$--orbits in $X$.

We suppose throughout this article that $X/G$ is connected. Now we can state
our main result.

\begin{thm*}\label{Thm:Main}
The following are equivalent.
\begin{enumerate}
\item The gradient map $\mu_\lie{p}$ locally almost separates the $K$--orbits.
\label{main:item:1}
\item The gradient map $\mu_\lie{p}$ almost separates the
$K$--orbits in $X$.\label{main:item:2}
\item The minimal parabolic subgroup $Q_0$ of $G$ has an open orbit in
$X$.\label{main:item:3}
\end{enumerate}
\end{thm*}

Hence, Theorem~\ref{Thm:Main} gives a sufficient condition on the $G$--action
for $\mu_\lie{p}$ to induce a map $X/K\to\lie{p}/K$ whose fibers are discrete,
while on the other hand the gradient map yields a criterion for $X$ to be
spherical. Moreover we see that sphericity is independent of the particular
choice of $\mu_\lie{p}$, i.\,e.\ if one gradient map for the $G$--action on $X$
generically separates the $K$--orbits in $X$, then this is true for every
gradient map.

Let us outline the main ideas of the proof. First we observe that $X$ contains
an open $Q_0$--orbit if and only if $(G/Q_0)\times X$ contains an open
$G$--orbit with respect to the diagonal action of $G$. The gradient map
$\mu_\lie{p}$ on $X$ induces a gradient map $\wt{\mu}_\lie{p}$ on
$(G/Q_0)\times X$. Now we are in a situation where we can apply the methods
introduced in~\cite{HeSchw}. These allow us to show that open $G$--orbits
correspond to isolated minimal $K$--orbits of the norm squared of
$\wt{\mu}_\lie{p}$. In order to relate the property that $\mu_\lie{p}$
locally almost separates the $K$--orbits to the existence of an isolated
minimal $K$--orbit, we need the following result. We consider the restriction
$\mu_\lie{p}|_{K\cdot x}\colon K\cdot x\to K\cdot\mu_\lie{p}(x)$ which is a
smooth fiber bundle with fiber $K_{\mu_\lie{p}(x)}/K_x$. In the special case
$G=K^\mbb{C}$ it is proven in~\cite{GS} that for generic $x$ the fiber
$K_{\mu_\lie{p}(x)}/K_x$ is a torus. As a generalization we prove the following
proposition, which also allows us to extend the notion of ``$K$--spherical''
defined in~\cite{HuWu} to actions of real-reductive groups.

\begin{prop*}
Let $x\in X$ be generic and choose a maximal Abelian subspace $\lie{a}$ of
$\lie{p}$ containing $\mu_\lie{p}(x)$. Then the orbits of the centralizer
$\mathcal{Z}_K(\lie{a})$ of $\lie{a}$ in $K$ are open in $K_{\mu_\lie{p}(x)}
/K_x$.
\end{prop*}

These arguments yield the existence of an open $Q_0$--orbit under the
assumption that $\mu_\lie{p}$ locally almost separates the $K$--orbits. For the
other direction we apply the shifting technique for gradient maps.

Notice that our proof of Brion's theorem is different from the ones
in~\cite{Bri} and~\cite{HuWu}. In particular, for every generic element $x\in
X$ we construct a minimal parabolic subgroup $Q_0$ of $G$ such that $Q_0\cdot
x$ is open in $X$.

At present we do not know whether a spherical $G$--gradient manifold does only
contain a finite number of $G$-- and $Q_0$--orbits (which is true in the
complex-algebraic situation). These and other natural open questions will be
addressed in future works.

\section{Gradient manifolds}

In this section we review the necessary background on $G$--gradient manifolds
and gradient maps. We then define what it means that a gradient map locally
almost separates the orbits of a maximal compact subgroup of $G$ and discuss
several examples where this can be shown to be true.

\subsection{The gradient map}

Here we recall the definition of the gradient map. For a detailed discussion we
refer the reader to \cite{HeSchw}.

Let $U$ be a compact Lie group and $U^\mbb{C}$ its universal
complexification (see~\cite{Hochschild}). We assume that $Z$ is a K\"ahler
manifold with a holomorphic action of $U^\mbb{C}$ such that the K\"ahler form
is invariant under the action of the compact real form $U$ of $U^\mbb{C}$. We
assume furthermore that the action of $U$ is Hamiltonian, i.\,e.\ that there
exists a moment map $\mu\colon Z\to\lie{u}^*$, where $\lie{u}^*$ is the dual of
the Lie algebra of $U$. We require $\mu$ to be real-analytic and
$U$--equivariant, where the action of $U$ on $\lie{u}^*$ is the coadjoint
action.

The complex reductive group $U^\mbb{C}$ admits a Cartan involution $\theta
\colon U^\mbb{C}\to U^\mbb{C}$ with fixed point set $U$. The $-1$-eigenspace of
the induced Lie algebra involution equals $\im\lie{u}$. We have an induced
Cartan decomposition, i.\,e.\ the map $U\times\im\lie{u}\to U^\mbb{C}$,
$(u,\xi)\mapsto u\exp(\xi)$ is a diffeomorphism. Let $G$ be a $\theta$-stable
closed real subgroup of $U^\mbb{C}$ with only finitely many connected
components. Equivalently, we assume that $G$ is a closed subgroup of
$U^\mbb{C}$, such that the Cartan decomposition restricts to a diffeomorphism
$K\times\lie{p}\to G$, where $K:=G\cap U$ and $\lie{p}:=\lie{g}\cap\im\lie{u}$.
In this paper such a group $G=K\exp(\lie{p})$ is called
{\emph{real-reductive}}. Note that $U^\mbb{C}$ itself is an example for such a
subgroup $G$ of $U^\mbb{C}$.

Let $X$ be a $G$--invariant real-analytic submanifold of $Z$ such that $X/G$ is
connected. We identify $\lie{u}$ with $\lie{u}^*$ by a $U$--invariant inner
product $\langle\cdot,\cdot\rangle$ on $\lie{u}$. Moreover we identify
$\lie{u}$ and $\im\lie{u}$ by multiplication with $\im$. Then the moment map
$\mu\colon Z\to\lie{u}^*$ restricts to a real-analytic map $\mu_\lie{p}\colon
X\to\lie{p}$ which is defined by $\bigl\langle\mu_\lie{p}(x),
\xi\bigr\rangle=\mu(x)(-\im\xi)$ for $\xi\in\lie{p}$. We call $\mu_\lie{p}$ a
\emph{$G$-gradient map on $X$} and we say that $X$ is a \emph{$G$-gradient
manifold}. Note that $\mu_\lie{p}$ is $K$--equivariant with respect to the
adjoint action of $K$ on $\lie{p}$. In the special case $G=U^\mbb{C}$, the
gradient map coincides with the moment map up to the identification of
$\lie{u}^*$ with $\im\lie{u}$.

In this paper, we consider real-analytic gradient maps which \emph{locally
almost separate the $K$--orbits}. By this, we mean that there exists a
$K$--invariant open subset $\Omega$ of $X$ such that the following equivalent
conditions are satisfied.
\begin{enumerate}
\item $K\cdot x$ is open in $\mu_\lie{p}^{-1}\bigl(K\cdot\mu_\lie{p}(x)\bigr)$
for all $x\in\Omega$.
\item $K_{\mu_\lie{p}(x)}\cdot x$ is open in
$\mu_\lie{p}^{-1}\bigl(\mu_\lie{p}(x)\bigr)$ for all $x\in \Omega$.
\item The induced map $\ol{\mu}_\lie{p}\colon\Omega/K\to\lie{p}/K$ has discrete
fibers.
\end{enumerate}

If $\Omega=X$, we say that $\mu_\lie{p}$ \emph{almost separates the
$K$--orbits}. We will show later that the set $\Omega$ on which $\mu_\lie{p}$
almost separates the $K$--orbits can always be chosen to be $X$, i.\,e.\
$\mu_\lie{p}$ separates locally almost the $K$--orbits if and only if
$\mu_\lie{p}$ almost separates them. If
$\mu_\lie{p}^{-1}\bigl(K\cdot\mu_\lie{p}(x)\bigr)= K\cdot x$ for all $x\in X$,
then we say that $\mu_\lie{p}$ \emph{globally separates the $K$--orbits}.

\begin{lem}\label{Lem:OpenGOrbits}
Suppose that $\mu_\lie{p}\colon X\to\lie{p}$ locally almost separates the
$K$--orbits. Then $G$ has an open orbit in $X$.
\end{lem}

\begin{proof}
By assumption there exists a $K$--invariant open subset $\Omega\subset X$ such
that $\mu_\lie{p}^{-1}\bigl(\mu_\lie{p}(x)\bigr)^0$ $\subset K\cdot x$ holds
for all $x\in\Omega$. Since $\mu_\lie{p}$ is real-analytic, we find a point
$x\in\Omega$ such that $\mu_\lie{p}$ has maximal rank in $x$. We conclude from
Lemma~5.1 in~\cite{HeSchw} that $(\lie{p}\cdot x)^\perp=T_x\mu_\lie{p}^{-1}
\bigl(\mu_\lie{p}(x)\bigr)\subset\lie{k}\cdot x$ and thus obtain
\begin{equation*}
T_xX=(\lie{p}\cdot x)\oplus(\lie{p}\cdot x)^\perp\subset(\lie{p}\cdot x)+
(\lie{k}\cdot x)=\lie{g}\cdot x,
\end{equation*}
which means that $G\cdot x$ is open in $X$.
\end{proof}

\subsection{Examples}\label{Section:Ex}

In general, it is very difficult to verify directly that a $G$--gradient map
separates (locally almost) the $K$--orbits. In this subsection we give some
examples of situations where this can be done.

\begin{ex}
The connected group $G=K\exp(\lie{p})$ acts on itself by left multiplication.
The standard gradient map for this action is given by $\mu_\lie{p}\colon G\to
\lie{p}$, $\mu_\lie{p}\bigl(k\exp(\xi)\bigr)=\Ad(k)\xi$. Let $x_0=k_0
\exp(\xi_0)\in G$ be given. One checks directly that $\mu_\lie{p}^{-1}\bigl(
\mu_\lie{p}(x_0)\bigr)=x_0K$. Hence, $\mu_\lie{p}$ locally almost separates the
$K$--orbits if and only if there exists a $K$--invariant open subset $\Omega
\subset G$ such that $xK=Kx$ for all $x\in\Omega$. We claim that this is the
case if and only if $\lie{p}^K=\lie{p}$.

Suppose that $xK=Kx$ holds for all $x$ in a $K$--invariant open subset
$\Omega\subset G$. This means that the fixed point set $(G/K)^K$ has
non-empty interior. Since $G/K$ is $K$--equivariantly diffeomorphic to
$\lie{p}$ with the adjoint $K$--action, we see that $\lie{p}^K$ has non-empty
interior and thus $\lie{p}^K=\lie{p}$.

Conversely, if $\lie{p}^K=\lie{p}$, then we have for every $x=k\exp(\xi)\in G$
that $Kx=K\exp(\xi)=\exp(\xi)K=xK$ holds.
\end{ex}

\begin{ex}\label{Ex:TotallyReal}
We describe a class of totally real $G$--gradient manifolds where $\mu_\lie{p}$
locally almost separates the $K$--orbits.

Let $(Z,\omega)$ be a K\"ahler manifold endowed with a holomorphic
$U^\mbb{C}$--action such that the $U$--action is Hamiltonian with moment map
$\mu\colon Z\to\lie{u}^*$. Suppose that the action is defined over $\mbb{R}$ in
the following sense. There exists an antiholomorphic involutive automorphism
$\sigma\colon U^\mbb{C}\to U^\mbb{C}$ with $\sigma\theta=\theta\sigma$ and
there is an antiholomorphic involution $\tau\colon Z\to Z$ with $\tau^*\omega
=-\omega$ and $\tau(g\cdot z)=\sigma(g)\cdot\tau(z)$ for all $g\in U^\mbb{C}$
and all $z\in Z$. Consequently, the fixed point set $X:=Z^\tau$ is a Lagrangian
submanifold of $Z$ and the compatible real form $G=K\exp(\lie{p})=
(U^\mbb{C})^\sigma$ acts on $X$. Let $\mu_\lie{p}\colon X\to\lie{p}$ be the
$K$--equivariant gradient map induced by $\mu$.

We claim that if $\mu$ locally almost separates the $U$--orbits in $Z$, then
$\mu_\lie{p}$ locally almost separates the $K$--orbits in $X$. This claim is a
consequence of the following three observations:
\begin{enumerate}
\item If $\mu$ locally almost separates the $U$--orbits, then $\mu$ separates
all
the $U$--orbits in $Z$ (see~\cite{HuWu}).
\item Since $X$ is Lagrangian, we see that $\mu_\lie{k}|_X\equiv0$, where
$\mu_\lie{k}$ denotes the moment map for the $K$--action on $Z$. Note that
under our identification we have $\mu=\mu_\lie{k}+\mu_\lie{p}$.
\item For every $x\in X$ the orbit $K\cdot x$ is open in $(U\cdot x)\cap X$.
\end{enumerate}
\end{ex}

Locally injective gradient maps separate locally almost the $K$--orbits. A
class of $G$--gradient manifolds for which $\mu_\lie{p}$ is locally injective
is described in the following example.

\begin{ex}
Let $Z=U/K$ be a Hermitian symmetric space of the compact type, and let $G=K
\exp(\lie{p})$ be a Hermitian real form of $U^\mbb{C}$. Then $Z$ is a
$G$--gradient manifold and every gradient map $\mu_\lie{p}\colon Z\to\lie{p}$
is locally injective. Consequently, $\mu_\lie{p}$ separates locally almost the
$K$--orbits in $Z$.

We will elaborate a little bit on further properties of $\mu_\lie{p}\colon Z\to
\lie{p}$. Let $\tau\colon Z\to Z$ be the holomorphic symmetry which fixes the
base point $z_0=eK$. Then we have $Z^\tau=\mu_\lie{p}^{-1}(0)$. Moreover, one
can show that $Z^\tau$ is a $K$--invariant closed complex submanifold of $Z$
and that every $K$--orbit in $Z^\tau$ is open in $Z^\tau$. Furthermore,
$K^\mbb{C}$ acts on $Z^\tau$ and we have $K^\mbb{C}\cdot z=K\cdot z$ if and
only if $z\in Z^\tau$ holds. Finally, note that $\mu_\lie{k}$ separates all
$K$--orbits in $Z$.
\end{ex}

\section{Spherical gradient manifolds and coadjoint orbits}

As we have remarked above it is very hard to verify directly if a given
gradient map defined on $X$ separates the $K$--orbits. The main result of this
paper states that this is true if and only if $X$ is a spherical gradient
manifold. Hence, this is independent of the particular choice of a gradient map
$\mu_\lie{p}$.

In this section we give the definition of spherical gradient manifolds. For
this we first review the definition of minimal parabolic subgroups. After that,
we discuss the orbits of the adjoint $K$--action on $\lie{p}$ which are the
right analogues of complex flag varieties.

We continue the notation of the previous section: Let $G=K\exp(\lie{p})$ be a
closed compatible subgroup of $U^\mbb{C}$ and let $X$ be a real-analytic
$G$--gradient manifold with $K$--equivariant real-analytic gradient map
$\mu_\lie{p}\colon X\to\lie{p}$.

\subsection{Minimal parabolic subgroups}

For more details and complete proofs of the material presented here we refer
the reader to Chapter~VII in~\cite{Kn}.

Since $G=K\exp(\lie{p})$ is invariant under the Cartan involution $\theta$ of
$U^\mbb{C}$, the same holds for its Lie algebra $\lie{g}=\lie{k}\oplus\lie{p}$.
Consequently $\lie{g}$ is reductive, i.\,e.\ $\lie{g}$ is the direct sum of its
center and of the semi-simple subalgebra $[\lie{g},\lie{g}]$.

Let $\lie{a}$ be a maximal Abelian subalgebra of $\lie{p}$ and let $\lie{g}=
\lie{g}_0\oplus\bigoplus_{\lambda\in\Lambda}\lie{g}_\lambda$ be the associated
restricted root space decomposition. The centralizer $\lie{g}_0$ of $\lie{a}$
in $\lie{g}$ is $\theta$--stable with decomposition $\lie{g}_0=\lie{m}\oplus
\lie{a}$ where $\lie{m}=\mathcal{Z}_\lie{k}(\lie{a})$. On the group level we
define $M:=\mathcal{Z}_K(\lie{a})$.

Let us fix a choice $\Lambda^+$ of positive restricted roots. Then we obtain
the nilpotent subalgebra $\lie{n}:=\bigoplus_{\lambda\in\Lambda^+}
\lie{g}_\lambda$. Let $A$ and $N$ be the analytic subgroups of $G$ with Lie
algebras $\lie{a}$ and $\lie{n}$, respectively. Then $AN\subset G$ is a
simply-connected solvable closed subgroup of $G$, isomorphic to the semi-direct
product $A\ltimes N$. One checks directly that $M$ stabilizes each restricted
root space $\lie{g}_\lambda$; together with the compactness of $M$ this implies
that $Q_0:=MAN$ is a closed subgroup of $G$.

Every subgroup of $G$ which is conjugate to $Q_0=MAN$ is called a \emph{minimal
parabolic subgroup}. A subgroup $Q\subset G$ is called \emph{parabolic} if
it contains a minimal parabolic subgroup.

\begin{rem}
The notion of parabolic subgroups of $G$ is independent of the choices made
during the construction of $Q_0$.
\end{rem}

\begin{ex}
For $\xi\in\lie{p}$ the group $Q:=\bigl\{g\in G;\ \lim_{t\to-\infty}
\exp(t\xi)g\exp(-t\xi)\text{ exists in }G\bigr\}$ is a parabolic subgroup
of $G$. It is a minimal parabolic subgroup if and only if $\xi$ is regular,
i.\,e.\ if and only if $K_\xi=M$.
\end{ex}

If the group $G$ is complex-reductive and connected, then minimal parabolic
subgroups of $G$ are the same as Borel subgroups. This motivates the following

\begin{defn}
We call the $G$--gradient manifold $X$ \emph{spherical} if a minimal parabolic
subgroup of $G$ has an open orbit in $X$.
\end{defn}

Note that $X$ is spherical if and only if $Q_0=MAN$ has an open orbit in $X$.

\begin{ex}
Let $G$ be a real form of $U^\mbb{C}$ and let $X\subset Z$ be a totally real
$G$--stable subma\-nifold with $\dim_\mbb{R}X=\dim_\mbb{C}Z$. If $Z$ is
$U^\mbb{C}$--spherical, then $X$ is $G$--spherical in the above sense. This can
be seen as follows. Since $Q_0^\mbb{C}$ is a parabolic subgroup of
$U^\mbb{C}=G^\mbb{C}$ and since $Z$ is spherical, $Q_0^\mbb{C}$ has an open
orbit in $Z$. Since $X$ is maximally totally real, $X$ cannot be contained in
the complement of the open $Q_0^\mbb{C}$--orbit in $Z$, hence we find a point
$x\in X$ such that $Q_0^\mbb{C}\cdot x$ is open in $Z$. Moreover, $Q_0\cdot x$
is open in $(Q_0^\mbb{C}\cdot x)\cap X$, which implies that $X$ is spherical.
\end{ex}

\begin{ex}
As a special case of the above example we note that weakly symmetric spaces are
spherical gradient manifolds. More precisely, let $G^\mbb{C}$ be connected
complex-reductive and let $L^\mbb{C}$ be a complex-reductive compatible
subgroup of $G^\mbb{C}$. Let $G$ be a connected compatible real form of
$G^\mbb{C}$ such that $L:=L^\mbb{C}\cap G$ is a compact real form of
$L^\mbb{C}$. According to Theorem~3.11 in~\cite{St} the homogeneous manifold
$X=G/L$ is a $G$--gradient manifold. By a result of Akhiezer and Vinberg
(\cite{AkhVin}, compare also Chapter~12.6 in~\cite{Wo}) $X=G/L$ is weakly
symmetric if and only if the affine variety $G^\mbb{C}/L^\mbb{C}$ is spherical.
This implies that if $X=G/L$ is weakly symmetric, then it is a spherical
$G$--gradient manifold. The converse is false as the next example shows.
\end{ex}

\begin{ex}
Let $U$ be connected. A special case of Example~\ref{Ex:TotallyReal} is the
case that $Z=U^\mbb{C}$ and $\tau=\sigma=\theta$. Then we have $G=X=U$. Note
that $\mu_\lie{p}\equiv0$ separates the $K$--orbits in $X$ since $X$ is
$K$--homogeneous while in general $\mu$ does not separate the $U$--orbits in
$Z$. Note also that $Q_0=G$ is the only minimal parabolic subgroup of $G$ and
that $G$ itself is the only subgroup of $G$ having an open orbit in
$X$. This explains the necessity to consider minimal parabolic subgroups
instead of maximal connected solvable subgroups (which are maximal tori in $G$
in this example).
\end{ex}

\subsection{Coadjoint orbits}\label{subsection:coadjoint}

A class of examples of gradient manifolds is given by coadjoint orbits
(see~\cite{HeSt1}). Let $\alpha\in\lie{u}^*$ and let $Z=U\cdot\alpha$ be the
coadjoint orbit of $\alpha$. Identifying $\lie{u}^*$ with $\im\lie{u}$ as
before, $\alpha$ corresponds to an element $\xi\in\im\lie{u}$ and $Z$
corresponds to the orbit of $\xi$ of the adjoint action of $U$ on $\im\lie{u}$.
Let $P:=\bigl\{g\in U^\mbb{C};\ \lim_{t\to-\infty}\exp(t\xi)g
\exp(-t\xi)\text{ exists in }U^\mbb{C}\bigr\}$ denote the parabolic subgroup
of $U^\mbb{C}$ associated to $\xi$. Then the map $Z\to U^\mbb{C}/P$,
$u\cdot\xi\mapsto uP$, is a real analytic isomorphism. In particular it defines
a complex structure and a holomorphic $U^\mbb{C}$--action on $Z$. The reader
should be warned that this $U^\mbb{C}$--action is not the adjoint action. The
form $\omega\bigl(\eta_Z(\alpha),\zeta_Z(\alpha)\bigr)=
-\alpha\bigl([\eta,\zeta] \bigr)$ defines a $U$--invariant K\"ahler form
on $Z=U\cdot\alpha$ such that the map $\mu\colon Z\to\lie{u}^*$,
$\mu(u\cdot\alpha)=-\Ad(u)\alpha$, is a moment map on $Z$. Identifying $Z$ with
$U/U_\xi$ where $U_\xi$ denotes the centralizer of $\xi$ in $U$, the
gradient map with respect to the action of $U^\mbb{C}$ on $Z$ is given by
$\mu_{\im\lie{u}}\colon U/U_\xi\to\im\lie{u}$,
$uU_\xi\mapsto-\Ad(u)\xi$. The $U^\mbb{C}$--action on $U\cdot\xi\cong
U^\mbb{C}/P$ induces a $G$--action on $U\cdot\xi$.

\begin{prop}[\cite{HeSt1}]\label{Prop:CoadjointOrbits}
If $\xi\in\lie{p}$, then $X:=K\cdot\xi=G\cdot\xi$ is a Lagrangian submanifold
of $Z\cong U\cdot\xi$.
\end{prop}

The $G$-isotropy at $\xi$ is given by the parabolic subgroup $Q:=P\cap G$ of
$G$, so $G\cdot\xi$ is isomorphic to $G/Q$ and to $K/K_\xi$ if $\xi\in\lie{p}$.
Note also that $G/Q$ is a compact $G$--invariant submanifold of $U^\mbb{C}/P$
and in particular a $G$-gradient manifold with gradient map $\mu_\lie{p}\colon
K/K_\xi\to\lie{p}$, $\mu_\lie{p}(kK_\xi)=-\Ad(k)\xi$.

\begin{ex}
Consider the action of $G={\rm{SL}}(2,\mbb{R})$ on projective space
$Z=\mbb{P}_1(\mbb{C})$ induced by the standard representation of $G$ on
$\mbb{C}^2$. Note that $G$ is a compatible subgroup of
$U^\mbb{C}={\rm{SL}}(2,\mbb{C})$ where $U={\rm{SU}}(2)$. Moreover, $Z$ can be
realized as the coadjoint orbit $U^\mbb{C}/B$ where $B$ is the Borel subgroup
$B=\left\{
\left(
\begin{array}{cc}
  z & w \\
  0 & z^{-1} \\
\end{array}%
\right);z\in\mbb{C}^*, w\in\mbb{C}
\right\}$. Then $Z$ can be viewed as a $2$-sphere in the $3$-dimensional space
$\im\lie{u}$. The gradient map $\mu_\lie{p}$ is the projection onto the
$2$-dimensional subspace $\lie{p}$ of $\im\lie{u}$. The action of $K$ on
$\im\lie{u}$ is given by rotation around the axes perpendicular to $\lie{p}$.
We observe that $\mu_\lie{p}$ almost separates the $K$--orbits, but that it
does not separate all $K$--orbits. This corresponds to the fact that there
exist two open orbits with respect to the action of a minimal parabolic
subgroup of $G$.
\end{ex}

If $G=U^\mbb{C}$ is complex reductive and acts algebraically on a connected
algebraic variety $Z$, then the fibers of the moment map $\mu$ are connected
(\cite{HeHu}). Also, if $Z$ is spherical, then $\mu$ globally separates the
$U$--orbits. The example above shows that one cannot expect $\mu_\lie{p}$ to
separate the $K$--orbits globally for actions of real-reductive groups due to
the non-connectedness of the $\mu_\lie{p}$--fibers. Moreover, in the complex
case an open orbit of a Borel subgroup is unique and dense in $Z$ while this is
no longer true for real-reductive groups.

\section{The generic fibers of the restricted gradient map}

By equivariance, the moment map $\mu\colon Z\to\lie{u}^*$ maps each orbit
$U\cdot z$ onto the orbit $U\cdot\mu(z)\subset\lie{u}^*$. Moreover, the
restriction $\mu|_{U\cdot z}\colon U\cdot z\to U\cdot\mu(z)$ is a smooth fiber
bundle with fiber $U_{\mu(z)}/U_z$. Theorem~26.5 in~\cite{GS} states that
generically these fibers are tori; in~\cite{HuWu} this theorem is applied to
characterize coisotropic $U$--actions.

In this section we generalize these results in our context. Let $x\in X$ and
let $\lie{a}$ be a maximal Abelian subspace of $\lie{p}$ with
$\mu_\lie{p}(x)\in\lie{a}$. Our goal is to prove that generically the group
$M=\mathcal{Z}_K(\lie{a})$ has an open orbit in the fiber
$K_{\mu_\lie{p}(x)}/K_x$ of $\mu_\lie{p}\colon K\cdot x\to
K\cdot\mu_\lie{p}(x)$. For this we first have to discuss the notion of generic
elements in $X$.

\subsection{Generic elements}

There are several natural definitions of generic elements $x\in X$. We could
require that the $K$--orbit through $x$ has maximal dimension, or that the
$K$--orbit through $\mu_\lie{p}(x)$ has maximal dimension in $\mu_\lie{p}(X)$,
or that the rank of $\mu_\lie{p}$ in $x$ is maximal. It will turn out that we
need all three properties.

\begin{defn}\label{Defn:generic}
The element $x\in X$ is called generic if
\begin{enumerate}
\item the dimension of $K\cdot x$ is maximal,
\item the rank of $\mu_\lie{p}$ in $x$ is maximal, and
\item the dimension of $K\cdot\mu_\lie{p}(x)$ is maximal in $\mu_\lie{p}(X)$.
\end{enumerate}
We write $X_{\gen}$ for the set of generic elements in $X$.
\end{defn}

\begin{rem}
In the complex case we have $\rk_z\mu=\dim U\cdot z$; hence, condition~(2) in
Definition~\ref{Defn:generic} is superfluous in this case.
\end{rem}

For the following lemma we need the analyticity of $\mu_\lie{p}$ and of the
$K$--action on $X$.

\begin{lem}
The set $X_{\gen}$ is $K$--invariant, open and dense in $X$.
\end{lem}

\begin{proof}
Since $X/G$ is connected, the same is true for $X/K$. It is then a well-known
consequence of the Slice Theorem that the set of points $x\in X$ such that
$K\cdot x$ has maximal dimension is open and dense in $X$ (see Theorem~3.1,
Chapter~IV in~\cite{Br}). Since $\mu_\lie{p}\colon X\to\lie{p}$ is
real-analytic, its maximal rank set is also open and dense. Hence, $X':=
\{x\in X;\ \dim K\cdot x, \rk_x\mu_\lie{p}\text{ maximal}\}$ is open and dense
in $X$.

We prove the lemma by showing that $X'\setminus X_{\gen}$ is analytic in $X'$.
Let $x_0\in X'\setminus X_{\gen}$. Since $\mu_\lie{p}$ has constant rank on
$X'$, there are local analytic coordinates $(x,U)$ around $x_0$ in $X$ and
$(y,V)$ around $\mu_\lie{p}(x_0)$ in $\mu_\lie{p}(X)$ in which $\mu_\lie{p}$
takes the form $\mu_\lie{p}(x_1,\dotsc,x_n)=(x_1,\dotsc,x_k)$. Since
$\mu_\lie{p}$ is $K$--equivariant, $U$ and $V$ may be chosen $K$--invariant.
Since $A:=\{y\in V;\ \dim K\cdot y\text{ is not maximal in $V$}\}$ is analytic
in $V$, we see that $(X'\setminus X_{\gen})\cap U=\mu_\lie{p}^{-1}(A)$ is
analytic in $U$. Thus $X'\setminus X_{\gen}$ is locally analytic in $X$ and
since it is closed, it is analytic.
\end{proof}

\subsection{The $M$--action on $\mu_\lie{p}^{-1}\bigl(\mu_\lie{p}(x)\bigr)$}

In this subsection we discuss the restricted gradient map
\begin{equation*}
\mu_\lie{p}|_{K\cdot x}\colon K\cdot x\to K\cdot\mu_\lie{p}(x).
\end{equation*}
Recall that this map is a smooth fiber bundle with fiber
$K_{\mu_\lie{p}(x)}/K_x$.

\begin{rem}
Let $\lie{a}$ be a maximal Abelian subspace of $\lie{p}$. Then we have
$M\subset K_{\mu_\lie{p}(x)}$ for every $x\in X$ with
$\mu_\lie{p}(x)\in\lie{a}$. Note that every $K$--orbit in $X$ intersects
$\mu_\lie{p}^{-1}(\lie{a})$.
\end{rem}

We will need the following lemma which extends the corresponding result
in~\cite{GS}.

\begin{lem}\label{Lem:TechnicalFact}
For every $x\in X_{\gen}$ we have
$[\lie{k}_{\mu_\lie{p}(x)},\lie{p}_{\mu_\lie{p}(x)}]\subset\lie{p}_x$.
\end{lem}

\begin{proof}
By definition of $X_{\gen}$ the set
\begin{equation*}
E:=\bigl\{(x,\xi,\eta)\in X_{\gen}\times\lie{k}\times\lie{p};\ \xi\in
\lie{k}_{\mu_\lie{p}(x)},\eta\in\lie{p}_{\mu_\lie{p}(x)}\bigr\}
\end{equation*}
is a linear subbundle of the trivial bundle
$X_{\gen}\times\lie{k}\times\lie{p}\to X_{\gen}$.

Let $\xi\in\lie{k}_{\mu_\lie{p}(x)}$ and $\eta\in\lie{p}_{\mu_\lie{p}(x)}$, and
let $x_t$ be a smooth curve in $X_{\gen}$ with $x_0=x$. Since $E\to X_{\gen}$
is locally trivial, we find a smooth curve $(x_t,\xi_t,\eta_t)$ in $E$ with
$\xi_0=\xi$ and $\eta_0=\eta$. Since $[\xi_t,\eta_t]\in
\lie{p}_{\mu_\lie{p}(x_t)}$ for all $t$ and since the inner product $\langle
\cdot,\cdot\rangle$ on $\lie{p}$ is induced by a $U$--invariant inner product
on $\lie{u}$, we conclude
\begin{equation*}
\bigl\langle\mu_\lie{p}(x_t),[\xi_t,\eta_t]\bigr\rangle
=-\bigl\langle\bigl[\xi_t,\mu_\lie{p}(x_t)\bigr],\eta_t\bigr\rangle=0
\end{equation*}
for all $t$. Differentiating and evaluating at $t=0$ yields
\begin{align*}
0&=\bigl\langle(\mu_\lie{p})_{*,x}\dot{x}_0,[\xi,\eta]\bigr\rangle+
\bigl\langle\mu_\lie{p}(x),[\dot{\xi}_0,\eta]\bigr\rangle+
\bigl\langle\mu_\lie{p}(x),[\xi,\dot{\eta}_0]\bigr\rangle\\
&=\bigl\langle(\mu_\lie{p})_{*,x}\dot{x}_0,[\xi,\eta]\bigr\rangle-
\bigl\langle\bigl[\eta,\mu_\lie{p}(x)\bigr],\dot{\xi}_0\bigr\rangle-
\bigl\langle\bigl[\xi,\mu_\lie{p}(x)\bigr],\dot{\eta}_0\bigr\rangle\\
&=\bigl\langle(\mu_\lie{p})_{*,x}\dot{x}_0,[\xi,\eta]\bigr\rangle=
g_x\bigl([\xi,\eta]_X(x),\dot{x}_0\bigr).
\end{align*}
Since $X_{\gen}$ is open, every tangent vector $v\in T_xX$ is of the form $v=
\dot{x}_0$ for some curve $x_t$ which implies $[\xi,\eta]_X(x)=0$, i.\,e.\
$[\xi,\eta]\in\lie{p}_x$.
\end{proof}

Now we are in the position to prove

\begin{prop}\label{Prop:TransitiveAction}
Suppose $x\in X_{\gen}\cap\mu_\lie{p}^{-1}(\lie{a})$. Then the orbit $M\cdot x$
is open in $\mu_\lie{p}^{-1}\bigl(\mu_\lie{p}(x)\bigr)\cap(K\cdot x)$.
\end{prop}

Let $x\in X_{\gen}\cap\mu_\lie{p}^{-1}(\lie{a})$ be given. In order to prove
Proposition~\ref{Prop:TransitiveAction} it suffices to show that the map
$\lie{m}\to\lie{k}_{\mu_\lie{p}(x)}/\lie{k}_x$ is surjective. For this we need
some information about $\lie{k}_{\mu_\lie{p}(x)}$ and $\lie{k}_x$; the idea is
of course to apply Lemma~\ref{Lem:TechnicalFact} which gives
\begin{equation*}
\bigl[[\lie{k}_{\mu_\lie{p}(x)},\lie{p}_{\mu_\lie{p}(x)}],
[\lie{k}_{\mu_\lie{p}(x)},\lie{p}_{\mu_\lie{p}(x)}]\bigr]\subset[\lie{p}_x,
\lie{p}_x]\subset\lie{k}_x.
\end{equation*}
Consequently we must determine $\lie{k}_{\mu_\lie{p}(x)}$,
$\lie{p}_{\mu_\lie{p}(x)}$ as well as their Lie brackets.

This is most conveniently done via the restricted root space decomposition
$\lie{g}=\lie{g}_0\oplus\bigoplus_{\lambda\in\Lambda}\lie{g}_\lambda$ with
respect to the maximal Abelian subspace $\lie{a}\subset\lie{p}$. The
centralizer $\lie{g}_0$ of $\lie{a}$ in $\lie{g}$ is stable under the Cartan
involution $\theta$ and decomposes as $\lie{g}_0=\lie{m}\oplus\lie{a}$ where
$\lie{m}=\Lie(M)$. For later use we note the following proposition which is
proven in Chapter~VI.5 of~\cite{Kn}.

\begin{prop}\label{Prop:SLTriples}
For each $\lambda\in\Lambda$ we write $\lie{a}_\lambda\subset\lie{a}$ for the
subspace generated by the elements $\bigl[\xi_\lambda,\theta(\xi_\lambda)
\bigr]$ where $\xi_\lambda\in\lie{g}_\lambda$. Then $\dim\lie{a}_\lambda=1$ and
$\lambda\bigl[\xi_\lambda,\theta(\xi_\lambda)\bigr]\not=0$ for every
$0\not=\xi_\lambda\in\lie{g}_\lambda$.
\end{prop}

In order to prove Proposition~\ref{Prop:TransitiveAction} we will first
describe the centralizers of $\mu_\lie{p}(x)$ in $\lie{k}$ and in $\lie{p}$.
For this we introduce the subset $\Lambda(x):=\bigl\{\lambda\in\Lambda;\
\lambda\bigl(\mu_\lie{p}(x)\bigr)=0 \bigr\}\subset\Lambda$. We also write
$\Lambda^+(x):=\Lambda(x)\cap\Lambda^+$.

\begin{rem}
If $\lambda\in\Lambda(x)$, then $-\lambda\in\Lambda(x)$. If
$\lambda_1,\lambda_2\in\Lambda(x)$ and $\lambda_1+\lambda_2\in
\Lambda$, then $\lambda_1+\lambda_2\in\Lambda(x)$.
\end{rem}

\begin{lem}\label{Lem:Centralizer}
\begin{enumerate}
\item The centralizer of $\mu_\lie{p}(x)$ in $\lie{g}$ is given by
$\lie{g}_0\oplus\bigoplus_{\lambda\in\Lambda(x)}\lie{g}_\lambda$.
\item We have $\lie{k}_{\mu_\lie{p}(x)}=\lie{m}\oplus\left\{\sum_{\lambda\in
\Lambda^+(x)}\bigl(\xi_\lambda+\theta(\xi_\lambda)\bigr);\ \xi_\lambda\in
\lie{g}_\lambda\right\}$.
\item We have $\lie{p}_{\mu_\lie{p}(x)}=\lie{a}\oplus\left\{\sum_{\lambda\in
\Lambda^+(x)}\bigl(\xi_\lambda-\theta(\xi_\lambda)\bigr);\ \xi_\lambda\in
\lie{g}_\lambda\right\}$.
\end{enumerate}
\end{lem}

\begin{proof}
In order to prove the first claim let $\xi=\xi_0+\sum_{\lambda\in\Lambda}
\xi_\lambda\in\lie{g}$ and calculate
\begin{equation*}
\bigl[\mu_\lie{p}(x),\xi\bigr]=\sum_{\lambda\in\Lambda}\lambda\bigl(
\mu_\lie{p}(x)\bigr)\xi_\lambda.
\end{equation*}
Hence, $\xi$ centralizes $\mu_\lie{p}(x)$ if and only if $\xi_\lambda=0$ for
all $\lambda\notin\Lambda(x)$.

The other two claims follow from $(1)$ together with the fact that
$\theta(\lie{g}_\lambda)=\lie{g}_{-\lambda}$ for all $\lambda\in\Lambda$.
\end{proof}

It remains to show that $\left\{\sum_{\lambda\in
\Lambda^+(x)}\bigl(\xi_\lambda+\theta(\xi_\lambda)\bigr);\ \xi_\lambda\in
\lie{g}_\lambda\right\}$ is contained in $\lie{k}_x$ because then
Lemma~\ref{Lem:Centralizer} implies that $\lie{m}\to\lie{k}_{\mu_\lie{p}(x)}
/\lie{k}_x$ is surjective which in turn proves
Proposition~\ref{Prop:TransitiveAction}.

\begin{lem}
We have $\left\{\sum_{\lambda\in\Lambda^+(x)} \bigl(\xi_\lambda+
\theta(\xi_\lambda)\bigr);\ \xi_\lambda\in\lie{g}_\lambda\right\}\subset
\lie{k}_x$.
\end{lem}

\begin{proof}
We will prove this lemma in three steps.

In the first step we prove
\begin{equation*}
\lie{p}^x:=\bigoplus_{\lambda\in\Lambda(x)}\lie{a}_\lambda\oplus
\left\{\sum_{\lambda\in\Lambda^+(x)}\bigl(\xi_\lambda-\theta(\xi_\lambda)\bigr)
;\ \xi_\lambda\in\lie{g}_\lambda\right\}\subset[\lie{k}_{\mu_\lie{p}(x)},
\lie{p}_{\mu_\lie{p}(x)}].
\end{equation*}

Let $\lambda\in\Lambda^+(x)$ and $\xi_\lambda\in\lie{g}_\lambda$. Then we have
$\xi_\lambda+\theta(\xi_\lambda)\in\lie{k}_{\mu_\lie{p}(x)}$, and we may choose
an element $\eta\in\lie{a}$ with $\lambda(\eta)\not=0$. Because of
\begin{equation*}
\xi_\lambda-\theta(\xi_\lambda)=-\frac{1}{\lambda(\eta)}\bigl[\xi_\lambda
+\theta(\xi_\lambda),\eta\bigr]\in[\lie{k}_{\mu_\lie{p}(x)},
\lie{p}_{\mu_\lie{p}(x)}]
\end{equation*}
we obtain $\left\{\sum_{\lambda\in\Lambda^+(x)}\bigl(\xi_\lambda-
\theta(\xi_\lambda)\bigr);\ \xi_\lambda\in\lie{g}_\lambda\right\}
\subset[\lie{k}_{\mu_\lie{p}(x)},\lie{p}_{\mu_\lie{p}(x)}]$.

Moreover,
\begin{equation*}
\bigl[\xi_\lambda,\theta(\xi_\lambda)\bigr]=-\frac{1}{2}
\bigl[\xi_\lambda+\theta(\xi_\lambda),\xi_\lambda-\theta(\xi_\lambda)\bigr]
\in[\lie{k}_{\mu_\lie{p}(x)},\lie{p}_{\mu_\lie{p}(x)}]
\end{equation*}
implies
$\lie{a}_\lambda\subset[\lie{k}_{\mu_\lie{p}(x)},\lie{p}_{\mu_\lie{p}(x)}]$.

The second step consists in showing
\begin{equation*}
\left\{\sum_{\lambda\in\Lambda^+(x)}\bigl(\xi_\lambda+\theta(\xi_\lambda)
\bigr);\ \xi_\lambda\in\lie{g}_\lambda\right\}\subset[\lie{p}^x,\lie{p}^x].
\end{equation*}

To see this, let $\lambda\in\Lambda^+(x)$ and
$0\not=\xi_\lambda\in\lie{g}_\lambda$ be arbitrary. Then we have
$\xi_\lambda-\theta(\xi_\lambda)\in\lie{p}^x$ and
$\bigl[\xi_\lambda,\theta(\xi_\lambda)\bigr]\in\lie{a}_\lambda$. Moreover,
Proposition~\ref{Prop:SLTriples} implies
$\lambda\bigl[\xi_\lambda,\theta(\xi_\lambda)\bigr]\not=0$, which gives
\begin{equation*}
\xi_\lambda+\theta(\xi_\lambda)=
\frac{1}{\lambda\bigl[\xi_\lambda,\theta(\xi_\lambda)\bigr]}
\Bigl[\bigl[\xi_\lambda,\theta(\xi_\lambda)\bigr],
\xi_\lambda-\theta(\xi_\lambda)\Bigr]\in[\lie{p}^x,\lie{p}^x].
\end{equation*}

In the last step we combine the results obtained so far with
Lemma~\ref{Lem:TechnicalFact} and arrive at
\begin{equation*}
\left\{\sum_{\lambda\in\Lambda^+(x)}\bigl(\xi_\lambda+\theta(\xi_\lambda)
\bigr);\ \xi_\lambda\in\lie{g}_\lambda\right\}\subset
[\lie{p}^x,\lie{p}^x]\subset\bigl[[\lie{k}_{\mu_\lie{p}(x)},
\lie{p}_{\mu_\lie{p}(x)}],[\lie{k}_{\mu_\lie{p}(x)},
\lie{p}_{\mu_\lie{p}(x)}]\bigr]\subset\lie{k}_x,
\end{equation*}
which was to be shown.
\end{proof}

Hence, the proof of Proposition~\ref{Prop:TransitiveAction} is finished.

\subsection{An equivalent condition of the separation property}

Proposition~\ref{Prop:TransitiveAction} allows us to formulate an equivalent
condition for $\mu_\lie{p}$ to separate locally almost the $K$--orbits which
generalizes the notion of $K$--spherical symplectic manifolds defined
in~\cite{HuWu}.

\begin{prop}\label{Prop:EquCondition}
The gradient map $\mu_\lie{p}$ locally almost separates the $K$--orbits if and
only if $\dim(\lie{p}\cdot x)^\perp=\dim M-\dim M_x$ for one (and then
every) $x\in X_{\gen}\cap\mu_\lie{p}^{-1}(\lie{a})$.
\end{prop}

\begin{proof}
Let us suppose first that $\mu_\lie{p}$ locally almost separates the
$K$--orbits. By definition, this means that there is an open $K$--invariant
subset $\Omega\subset X$ such that $\mu_\lie{p}^{-1}\bigl(
\mu_\lie{p}(x)\bigr)^0=K^0_{\mu_\lie{p}(x)}\cdot x$ for all $x\in\Omega$.

Since $X_{\gen}$ is dense, we find an element $x\in\Omega\cap X_{\gen}\cap
\mu_\lie{p}^{-1}(\lie{a})$. It follows from maximality of $\rk_x\mu_\lie{p}$
that $\mu_\lie{p}^{-1}\bigl(\mu_\lie{p}(x)\bigr)\cap X_{\gen}$ is a closed
submanifold of $X_{\gen}$. By Lemma~5.1 in \cite{HeSchw}, we obtain $\dim
\ker(\mu_\lie{p})_{*,x} =\dim(\lie{p}\cdot x)^\perp$. Hence, we conclude $\dim
K_{\mu_\lie{p}(x)}/K_x=\dim(\lie{p}\cdot x)^\perp$. Since by
Proposition~\ref{Prop:TransitiveAction} the orbit $M\cdot x$ is open in
$K_{\mu_\lie{p}(x)}\cdot x$, we finally obtain $\dim(\lie{p}\cdot x)^\perp=\dim
M/M_x=\dim M-\dim M_x$ which was to be shown.

In order to prove the converse let $x\in X_{\gen}\cap\mu_\lie{p}^{-1}(\lie{a})$
be given. Our assumption implies that $\mu_\lie{p}^{-1}\bigl(\mu_\lie{p}(x)
\bigr)$ is a closed submanifold of $X$ of dimension $\dim(\lie{p}\cdot x)^\perp
=\dim M-\dim M_x$. We conclude that $M\cdot x$ and hence $K_{\mu_\lie{p}(x)}
\cdot x$ are open in $\mu_\lie{p}^{-1}\bigl(\mu_\lie{p}(x)\bigr)$. Therefore we
have $\mu_\lie{p}^{-1}\bigl(\mu_\lie{p}(x)\bigr)^0=K^0_{\mu_\lie{p}(x)}\cdot
x$, which means that $\mu_\lie{p}$ separates the $K$--orbits in $X_{\gen}$.
\end{proof}

Let us note explicitly the following corollary of the proof of
Proposition~\ref{Prop:EquCondition}.

\begin{cor}
If $\mu_\lie{p}$ locally almost separates the $K$--orbits in $X$, then it almost
separates the $K$--orbits in the dense open set $X_{\gen}$.
\end{cor}

Consequently, if $\mu_\lie{p}$ locally almost separates the $K$--orbits in $X$,
then $\mu_\lie{p}$ induces a map $X_{\gen}/K\to\lie{p}/K\cong\lie{a}/W$ whose
fibers are discrete.

\section{Proof of the main theorem}\label{section:proof}

In the first subsection we review the shifting technique for gradient maps
which translates the problem of finding an open $Q_0$--orbit in $X$ into the
problem of finding an open $G$--orbit in the bigger gradient manifold
$X\times(K/M)$. Since $G$ is real-reductive, we may apply the techniques
developed in~\cite{HeSchw} to solve the second problem.

Afterwards, it remains to find an open $G$--orbit in $X\times(K/M)$ under the
assumption that $\mu_\lie{p}$ locally almost separates the $K$--orbits. This is
done in two steps: First we construct a special gradient map $\wt{\mu}_\lie{p}$
on $X\times(K/M)$ for which the set of global minima of
$\norm{\wt{\mu}_\lie{p}}^2$ can be controlled. This will then be essentially
used in the proof of existence of an open $Q_0$--orbit.

In the final subsection we prove the remaining implication
$(3)\Longrightarrow(2)$ in our main theorem: If the minimal parabolic subgroup
$Q_0$ has an open orbit in $X$, then $\mu_\lie{p}$ almost separates the
$K$--orbits.

\subsection{The shifting technique}

Since the minimal parabolic subgroup $Q_0=MAN$ is not compatible, we cannot
apply the theory developed in~\cite{HeSchw} in order to link the action of
$Q_0$ on $X$ with the theory of gradient maps. Therefore, we reformulate the
problem of finding an open $Q_0$--orbit in $X$ as the problem of finding an
open $G$--orbit in a larger manifold.

\begin{lem}\label{Lem:Reformulation}
Let $Q$ be a parabolic subgroup of $G$. Then $Q$ has an open orbit in $X$ if
and only if $G$ has an open orbit in $X\times (G/Q)$ with respect to the
diagonal action.
\end{lem}

\begin{proof}
Recall that the twisted product $G\times_{Q}X$ is by definition the quotient
space of $G\times X$ by the $Q$--action $q\cdot(g,x):=(gq^{-1},q\cdot x)$. We
denote the element $Q\cdot(g,x)\in G\times_QX$ by $[g,x]$. Then $G$ acts on
$G\times_QX$ by $g\cdot[h,x]:=[gh,x]$, and every $G$--orbit in $G\times_{Q}
X$ intersects $X\cong\bigl\{[e,x];\ x\in X\bigr\}$ in a $Q$--orbit. Thus, the
inclusion $X\hookrightarrow G\times_{Q}X$, $x\mapsto[e,x]$, induces a
homeomorphism $X/Q\cong (G\times_{Q}X)/G$. In particular, $Q$ has an open
orbit in $X$ if and only if $G$ has an open orbit in $G\times_{Q}X$.

The claim follows now from the fact that the map $G\times_{Q}X\to X
\times (G/Q)$, $[g,x]\mapsto(g\cdot x,gQ)$, is a $G$--equivariant
diffeomorphism with respect to the diagonal $G$--action on $X\times(G/Q)$. To
see this, it is sufficient to note that its inverse map is given by
$(x,gQ)\mapsto[g,g^{-1} \cdot x]$.
\end{proof}

The gradient map $\mu_\lie{p}$ on $X$ induces in a natural way a gradient map
on the product $\wt{X}:=X\times(G/Q)$ as follows. First recall
from Section~\ref{subsection:coadjoint} that $G/Q$ is a $G$--invariant
closed submanifold of an adjoint $U$--orbit of an element $\gamma\in\lie{p}$.
In particular $G/Q$ is isomorphic to $K/K_\gamma$ and is equipped with a
gradient map $kK_\gamma\mapsto-\Ad(k)\xi$. The gradient maps on $X$ and on
$K/K_\gamma$ induce a gradient map $\wt{\mu}_\lie{p}$ on $\wt{X}$, which is
given by the sum of those two gradient maps. Explicitly, we have
\begin{equation*}
\wt\mu_\lie{p}(x,kK_\gamma)=\mu_\lie{p}(x)-\Ad(k)\gamma.
\end{equation*}
Note that the choice of $\gamma\in\lie{p}$ depends only on the isotropy
$K_\gamma$. In particular, if $Q$ is a minimal parabolic subgroup of $G$, or
equivalently if $K_\gamma$ equals the centralizer $M$ of $\lie{a}$ in $K$, then
for every regular $\gamma\in\lie{p}$, the assignment $(x,kM)\mapsto
\mu_\lie{p}(x)-\Ad(k)\gamma$ defines a gradient map on $\wt{X}$.

\subsection{The shifted gradient map}

Our goal is to construct a gradient map on $\wt{X}=X\times(K/M)$ which enables
us to control the minima of the associated function $\|\wt\mu_\lie{p}\|^2$.

Let $\lie{a}_+$ denote the closed Weyl chamber in $\lie{a}$ associated to our
choice of positive restricted roots. We generalize an inequality in
\cite{HSchuetz} which is a consequence of Kostant's Convexity Theorem
(\cite{Kos}).

\begin{lem}\label{lem:kostant}
Let $\gamma,\xi\in\lie{a}_+$ and assume that $\xi$ is regular.
Then
\begin{equation*}
\norm{\Ad(k)\gamma-\xi}\geq\norm{\gamma-\xi}
\end{equation*}
for all $k\in K$. The inequality is strict for all $k\notin K_\gamma$.
\end{lem}

\begin{proof}
The $K$-invariance of the inner product implies
\begin{equation*}
\norm{\Ad(k)\gamma-\xi}^2-\norm{\gamma-\xi}^2=-2\cdot\bigl\langle\Ad(k)\gamma-
\gamma,\xi\bigr\rangle.
\end{equation*}
Let $\pi_\lie{a}$ denote the orthogonal projection of $\lie{p}$ onto $\lie{a}$.
Then $\bigl\langle\Ad(k)\gamma,\xi\bigr\rangle=
\bigl\langle\pi_\lie{a}(\Ad(k)\gamma),\xi\bigr\rangle$ and
$\pi_\lie{a}\bigl(\Ad(k)\gamma\bigr)$ is contained in the convex hull of the
orbit of the Weyl group $W:=\mathcal{N}_K(\lie{a})/\mathcal{Z}_K(\lie{a})$
through $\xi$ (\cite{Kos}). Since $K$ acts by unitary operators, we have
$\pi_\lie{a}\bigl(\Ad(k)\gamma\bigr)=\gamma$ if and only if $k\in K_\gamma$.
Therefore it suffices to show that
$\bigl\langle\Ad(w)\gamma-\gamma,\xi\bigr\rangle<0$ for all $w\in W$,
$w\notin W_\gamma$.

Let $\lambda$ be a simple restricted root and $\sigma_\lambda$ the
corresponding reflection. Then either $\sigma_\lambda(\gamma)=\gamma$ or
$\sigma_\lambda(\gamma)-\gamma=c\cdot\lambda$ for some $c<0$. Here we have
identified $\lambda\in\lie{a}^*$ with its dual in $\lie{a}$. Since $\xi$ is
regular, this implies
$\bigl\langle\sigma_\lambda(\gamma)-\gamma,\xi\bigr\rangle<0$ if
$\sigma_\lambda\notin W_\gamma$.

An arbitrary element $w\in W$ is of the form $w=\sigma_{\lambda_1}
\circ\dotsb\circ\sigma_{\lambda_k}$ for simple restricted roots
$\lambda_1,\dotsc,\lambda_k$. Then
\begin{align*}
\Ad(w)\gamma-\gamma=&\bigl(\sigma_{\lambda_1}\circ\dotsb\circ\sigma_{\lambda_k}
(\gamma)-\sigma_{\lambda_2}\circ\dotsb\circ\sigma_{\lambda_k}(\gamma)\bigr)\\
&+\bigl(\sigma_{\lambda_2}\circ\dotsb\circ\sigma_{\lambda_k}
(\gamma)-\sigma_{\lambda_3}\circ\dotsb\circ\sigma_{\lambda_k}(\gamma)\bigr)\\
&+\dotsb+\bigl(\sigma_{\lambda_k}(\gamma)-\gamma\bigr)
\end{align*}
is a linear combination of simple restricted roots with negative coefficients
and it equals $0$ if and only if $\sigma_{\lambda_j}\in\mathcal{W}_\gamma$ for
all $j$. Again, since $\xi$ is regular, this implies
$\bigl\langle\Ad(w)\gamma-\gamma,\xi\bigr\rangle<0$ for all $w\in W$,
$w\notin W_\gamma$.
\end{proof}

Since each $K$--orbit in $\lie{p}$ intersects $\lie{a}$ in an orbit of the Weyl
group, each $K$--orbit $K\cdot x$ in $X$ contains an $x_0$ with
$\mu_\lie{p}(x_0)\in\lie{a}_+$. Recall that each $\xi\in\lie{a}_+$ defines a
gradient map $\wt{\mu}_\lie{p}\colon\wt{X}\to\lie{p}$,
$\wt{\mu}_\lie{p}(x,kM)=\mu_\lie{p}(x)-\Ad(k)\xi$.

\begin{prop}\label{prop:gradientabb}
Let $x_0\in X_{\gen}$ with $\mu_\lie{p}(x_0)\in\lie{a}_+$. Then there exists a
regular $\xi\in\lie{a}_+$, such that
\begin{enumerate}
\item $(x_0,eM)$ is a global minimum of the function
$\norm{\wt{\mu}_\lie{p}}^2$.
\item If $(x,kM)\in\wt X$ is another global minimum of
$\norm{\wt{\mu}_\lie{p}}^2$, then $\mu_\lie{p}(x)=\Ad(k)\mu_\lie{p}(x_0)$.
\end{enumerate}
\end{prop}

\begin{proof}
If $\mu_\lie{p}(x_0)$ is regular, define $\xi:=\mu_\lie{p}(x_0)$. Then
$\norm{\wt{\mu}_\lie{p}(x_0,eM)}^2=0$ and $(x_0,eM)$ is a global minimum of
$\norm{\wt{\mu}_\lie{p}}^2$. If $(x,kM)$ is another global minimum, we have
$\mu_\lie{p}(x)-\Ad(k)\xi=0$ and the second claim follows.

Now assume that $\gamma:=\mu_\lie{p}(x_0)$ is singular. Let $\lambda_1,\dotsc,
\lambda_k$ be those simple restricted roots vanishing at $\gamma$. Let
$\lie{b}:=\bigl\{\eta\in\lie{a};\ \lambda_1(\zeta)=\ldots=\lambda_k(\eta)=
0\bigr\}$ be the subspace of $\lie{a}$ where these roots vanish. Let
$\lie{b}^\perp$ be the orthogonal complement of $\lie{b}$ in $\lie{a}$. Since
$x_0$ is regular, the orbit $K\cdot\gamma$ has maximal dimension in
$\mu_\lie{p}(X)$. Therefore $\mu_\lie{p}(X)\cap\lie{a}$ is contained in the
union of the finitely many subspaces of $\lie{a}$ where at least $k$ simples
restricted roots vanish. Choosing a regular element
$\xi\in\gamma+\lie{b}^\perp$ which is sufficiently near $\gamma$, we can assure
that $\gamma$ is the unique point in $\mu_\lie{p}(X)\cap\lie{a}_+$
with minimal distance to $\xi$.

Let $(x,kM)\in\wt{X}$ and let $l\in K$ with $\gamma':=\Ad(l)\mu_\lie{p}
(k^{-1}\cdot x)\in\lie{a}_+$. With Lemma~\ref{lem:kostant} and the definition
of $\xi$ we obtain
\begin{align*}
\norm{\wt{\mu}_\lie{p}(x,kM)}^2=\norm{\mu_\lie{p}(x)-\Ad(k)\xi}^2
&=\norm{\mu_\lie{p}(k^{-1}\cdot x)-\xi}^2\\
&\geq\norm{\gamma'-\xi}^2\geq\norm{\gamma-\xi}^2=
\norm{\wt{\mu}_\lie{p}(x_0,eM)}^2,
\end{align*}
so in particular $(x_0,eM)$ is a global minimum of $\norm{\wt{\mu}_\lie{p}}^2$.
Equality holds if and only if $\gamma'=\gamma$ and $l\in K_{\gamma'}=
K_{\gamma}$. The latter condition gives $\Ad(k)\gamma=\mu_\lie{p}(x)$.
\end{proof}

In Lemma~\ref{Lem:Reformulation}, we reformulated the property that a parabolic
subgroup $Q$ has an open orbit in $X$ as a property on the $G$-action on the
product $X\times(G/Q)$. Now, we translate the condition, that $\mu_\lie{p}$
locally almost separates the $K$-orbits to a suitable condition on the shifted
gradient map $\wt{\mu}_\lie{p}$ on the product $X\times(G/Q)$.

\begin{lem}\label{lemma:FaserQuotienten}
Let $\xi\in\lie{a}$ and let $\wt{\mu}_\lie{p}\colon\wt X\to\lie{p}$ be the
associated gradient map. Let $x_0\in X$ with
$\mu_\lie{p}(x_0)\in\lie{a}_+$ and set $\beta:=\mu_\lie{p}(x_0)-\xi=
\wt{\mu}_\lie{p}(x_0,eM)$. Then the inclusion
$\mu_\lie{p}^{-1}\bigl(\mu_\lie{p}(x_0)\bigr)\hookrightarrow
\wt{\mu}_\lie{p}^{-1}(\beta)$, $x\mapsto(x,eM)$, induces an injective
continuous map $\Phi\colon\mu_\lie{p}^{-1}\bigl(\mu_\lie{p}(x_0)\bigr)/M\to
\wt{\mu}_\lie{p}^{-1}(\beta)/K_\beta$. If $\xi$ is chosen such that the
conclusions of Proposition~\ref{prop:gradientabb} are satisfied, then $\Phi$ is
a homeomorphism.
\end{lem}

\begin{proof}
First note that the map
$\Phi\colon\mu_\lie{p}^{-1}\bigl(\mu_\lie{p}(x_0)\bigr)/M\to
\wt{\mu}_\lie{p}^{-1}(\beta)/K_\beta$ is well-defined since $M$ is contained
in $K_\beta$ and $K_{\mu_\lie{p}(x_0)}$ and since $\mu_\lie{p}$ and
$\wt{\mu}_\lie{p}$ are $K$--equivariant.

For injectivity, let $x,y\in\mu_\lie{p}^{-1}\bigl(\mu_\lie{p}(x_0)\bigr)$ with
$K_\beta\cdot(x,eM)=K_\beta\cdot(y,eM)$. The latter condition implies $M\cdot
x=M\cdot y$ since $K_\beta\cap M=M$. This shows injectivity.

Assume that $x_0\in\mu_\lie{p}^{-1}\bigl(\mu_\lie{p}(x_0)\bigr)$ satisfies the
conclusions of Proposition~\ref{prop:gradientabb} and let
$(x,kM)\in\wt{\mu}_\lie{p}^{-1}(\beta)$. Then $(x,kM)$ is a global minimum
of $\norm{\wt{\mu}_\lie{p}}^2$ which implies $\mu_\lie{p}(x)=
\Ad(k)\mu_\lie{p}(x_0)$. We conclude
$\beta=\wt{\mu}_\lie{p}(x,kM)=\mu_\lie{p}(x)-\Ad(k)\xi=\Ad(k)(\mu_\lie{p}(x_0)
-\xi)=\Ad(k)\beta$. This proves $k\in K_\beta$. Consequently $K_\beta\cdot
(x,kM)$ intersects $\mu_\lie{p}^{-1}\bigl(\mu_\lie{p}(x_0)\bigr)\times\{eM\}$
and surjectivity follows. Finally, the inclusion
$\mu_\lie{p}^{-1}\bigl(\mu_\lie{p}(x_0)\bigr)\hookrightarrow
\wt{\mu}_\lie{p}^{-1}(\beta)$ is continuous and proper, so $\Phi$ is continuous
and proper which implies that it is a homeomorphism.
\end{proof}

\subsection{Existence of an open $Q_0$--orbit}

Finally we are in the position to prove that $Q_0$ has an open orbit in $X$
given that $\mu_\lie{p}$ locally almost separates the $K$--orbits.

Let us fix a point $x_0\in X_{\gen}$ such that $\mu_\lie{p}(x_0)$ lies in the
closed Weyl chamber $\lie{a}_+$. By virtue of
Proposition~\ref{prop:gradientabb} we find a regular element $\xi\in\lie{a}_+$
such that $\wt{\mu}_\lie{p}\colon X\times(K/M)\to\lie{p}$, $(x,kM)\mapsto
\mu_\lie{p}(x)-\Ad(k)\xi$, is a $G$--gradient map and such that $\wt{x}_0:=
(x_0,eM)$ is a global minimum of $\norm{\wt{\mu}_\lie{p}}^2$. Let $Q_0=MAN$ be
the minimal parabolic subgroup of $G$ associated to $\xi$. Then we may
identify $K/M$ with $G/Q_0$ as gradient manifolds. Let
$\beta:=\mu_\lie{p}(x_0)-\xi$. By Lemma~\ref{lemma:FaserQuotienten} the
quotients $\mu_\lie{p}^{-1}\bigl(\mu_\lie{p}(x_0)\bigr)/M$ and
$\wt{\mu}_\lie{p}^{-1}(\beta)/K_\beta$ are homeomorphic. This implies that
$K_\beta\cdot\wt{x}_0$ is open in $\wt{\mu}_\lie{p}^{-1}(\beta)$.

As we have already seen in the proof of Lemma~\ref{Lem:OpenGOrbits}, it
suffices to prove $(\lie{p}\cdot\wt{x}_0)^\perp
\subset\lie{k}\cdot\wt{x}_0$, for then the orbit $G\cdot\wt{x}_0$ is open in
$X\times(G/Q_0)$ which in turn implies that $Q_0\cdot x_0$ is open in $X$. For
this we will show that $\wt{\mu}_\lie{p}$ has maximal rank in $\wt{x}_0$ as
follows. The image of $T_{x_0}X\oplus T_{eM}K/M$ under $(\wt{\mu}_\lie{p}
)_{*,\wt{x}_0}$ coincides with $(\mu_\lie{p})_{*,x_0}(T_{x_0}X)+[\lie{k},\xi]$.
Since $\xi$ is regular, we obtain
\begin{equation*}
[\lie{k},\xi]=\left\{\sum_{\lambda\in\Lambda^+}\bigl(\xi_\lambda-
\theta(\xi_\lambda)\bigr);\ \xi_\lambda\in\lie{g}_\lambda\right\}=\lie{a}^\perp.
\end{equation*}
We use the decomposition $T_xX=(\lie{k}\cdot x)\oplus(\lie{k}\cdot x)^\perp$
and note that $(\mu_\lie{p})_{*,x}$ maps $\lie{k}\cdot x$ into $\lie{a}^\perp$
for all $x$ in a neighborhood of $x_0$. Since moreover $\mu_\lie{p}$
locally almost separates the $K$--orbits, one would expect that
$(\mu_\lie{p})_{*,x_0}$ maps a subspace of $T_{x_0}X$ which is transversal to
$\lie{k} \cdot x_0$ onto a subspace of $\lie{p}$ which is transversal to
$\lie{a}^\perp$. This is the content of the following

\begin{lem}\label{Lem:TransversalImage}
Assume that $\mu_\lie{p}$ locally almost separates the $K$-orbits. For every
$x\in
X_{\gen}\cap\mu_\lie{p}^{-1}(\lie{a})$ we have
$(\mu_\lie{p})_{*,x}((\lie{k}\cdot x)^\perp)\cap\lie{a}^\perp=\{0\}$.
\end{lem}

\begin{proof}
Since $x$ is generic, there exists an open neighborhood $V\subset X$ of $x$
such that the rank of $\mu_\lie{p}$ is constant on $V$. We conclude that
$V\cap\mu_\lie{p}^{-1}(\lie{a})$ is a submanifold of $V$ and that the image
$\mu_\lie{p}\bigl(V\cap\mu_\lie{p}^{-1}(\lie{a})\bigr)$ is an open subset of
the linear subspace $\lie{b}:=\bigcap_{\lambda\in\Lambda(x)}\ker(\lambda)
\subset\lie{a}$. Moreover, we have $\mu_\lie{p}(V)$ is an open subset of
$K\cdot\lie{b}\cong K\times_{K_{\mu_\lie{p}(x)}}\lie{b}=(K/K_{\mu_\lie{p}(x)})
\times\lie{b}$.

Since $\mu_\lie{p}$ separates the $K$--orbits and since $x$ is generic, we have
$\ker(\mu_\lie{p})_{*,x}=(\lie{p}\cdot x)^\perp\subset\lie{k}\cdot x$ which
implies that $(\mu_\lie{p})_{*,x}$ is injective on $(\lie{k}\cdot x)^\perp$.
Consequently, $\mu_\lie{p}$ induces an injective immersion $V/K\to\lie{b}$,
therefore $(\mu_\lie{p})_{*,x}$ maps $(\lie{k}\cdot x)^\perp$ bijectively onto
$\lie{b}$. Since $\lie{b}\cap\lie{a}^\perp=\{0\}$, the claim follows.
\end{proof}

We conclude from Lemma~\ref{Lem:TransversalImage} that the image of
$(\wt{\mu}_\lie{p})_{*,\wt{x}_0}$ is given by $(\mu_\lie{p})_{*,x_0}
\bigl((\lie{k}\cdot x_0)^\perp\bigr)\oplus\lie{a}^\perp$. Since $x_0$ is
generic, the dimension of $(\mu_\lie{p})_{*,x}\bigl((\lie{k}\cdot x)^\perp
\bigr)$ is the same for all $x$ in a neighborhood of $x_0$. Furthermore, every
$K$--orbit in $X\times(K/M)$ intersects $X\times\{eM\}$, thus the rank of
$\wt{\mu}_\lie{p}$ is constant in a neighborhood of $\wt{x}_0$. Consequently,
the rank of $\wt{\mu}_\lie{p}$ must be maximal in $\wt{x}_0$. Together with the
fact that $K_\beta\cdot \wt{x}_0$ is open in $\wt{\mu}_\lie{p}^{-1}(\beta)$
this yields
\begin{equation*}
(\lie{p}\cdot\wt{x}_0)^\perp=T_{\wt{x}_0}\wt{\mu}_\lie{p}^{-1}(\beta)=
\lie{k}_\beta\cdot\wt{x}_0\subset\lie{k}\cdot\wt{x}_0.
\end{equation*}
Therefore we obtain $T_{\wt{x}_0}\wt{X}=\lie{p}\cdot\wt{x}_0\oplus(\lie{p}\cdot
\wt{x}_0)^\perp\subset\lie{p}\cdot\wt{x}_0+\lie{k}\cdot\wt{x}_0$ which shows
that $G\cdot\wt{x}_0$ is open in $\wt{X}$.

This proves the implication $(1)\Longrightarrow(3)$ of our our main theorem and
gives in addition a precise description of the set of open $Q_0$-orbits in $X$.

\begin{thm}
Suppose that $\mu_\lie{p}$ locally almost separates the $K$-orbits. Let $x_0\in
X_{\gen}\cap\mu_\lie{p}^{-1}(\lie{a}_+)$ be given, let $\xi$ be the element
from Proposition~\ref{prop:gradientabb}, and let $Q_0$ be the minimal parabolic
subgroup of $G$ associated to $\xi$. Then $Q_0\cdot x_0$ is open in $X$.
\end{thm}

The same method of proof gives the following

\begin{prop}\label{Prop:ProofSpecialCase}
Suppose that $\mu_\lie{p}\colon X\to\lie{p}$ locally almost separates the
$K$--orbits. Let $x\in X_{\gen}\cap\mu_\lie{p}^{-1}(\lie{a})$ and let $Q$ be
the parabolic subgroup of $G$ associated to $\beta:=\mu_\lie{p}(x)$. Then
$Q\cdot x$ is open in $X$.
\end{prop}

\begin{proof}
In order to show that $Q\cdot x$ is open in $X$, it suffices to show that
$G\cdot(x,eQ)$ is open in $X\times(G/Q)$. For this we note that $G/Q\cong
K/K_\beta$ as a $K$--manifold and that for the shifted gradient map
$\wt{\mu}_\lie{p}\colon X\times(K/K_\beta\to\lie{p}$,
$(x,kK_\beta)\mapsto\mu_\lie{p}(x)-\Ad(k)\beta$ the element $(x,eK_\beta)$ lies
in $\wt{\mathcal{M}}_\lie{p}$. Then the same arguments as above apply to show
that $G\cdot(x,eK_\beta)$ is open.
\end{proof}

\subsection{Proof of $(3)\Longrightarrow(2)$}

In this subsection we complete the proof of our main theorem by showing the
remaining non-trivial implication.

\begin{prop}
Suppose that $Q_0$ has an open orbit in $X$. Then $\mu_\lie{p}$ almost
separates the $K$--orbits.
\end{prop}

\begin{proof}
Let $x_0\in X$ be given. We must show that $K_{\mu_\lie{p}(x_0)}\cdot x_0$
is open in $\mu_\lie{p}^{-1}\bigl( \mu_\lie{p}(x_0)\bigr)$. Let
$\gamma:=\mu_\lie{p}(x_0)$ and let $Q$ be the parabolic subgroup of $G$
associated to $\gamma$. Recall that $G/Q\cong K/K_\gamma$ is a $G$-gradient
space with gradient map $kK_\gamma\mapsto -\Ad(k)\gamma$. Consider the shifted
gradient map $\wt{\mu_\lie{p}}\colon X\times(K\cdot\gamma)\to\lie{p}$,
$(x,kK_\gamma)\mapsto x-\Ad(k)\gamma$. Since the minimal parabolic subgroup
$Q_0$ has an open orbit in $X$, the same is true for $Q$. Hence $G$ has an open
orbit in $X\times(K/K_\gamma)$ by Lemma~\ref{Lem:Reformulation}.

By definition of $\gamma$, we have $\wt{\mu}_\lie{p}(x_0,\gamma)=0$. Consider
the set of semistable points $\mathcal{S}_G(\wt\mu_\lie{p}^{-1}(0))=\{\wt x\in
\wt X;\ \ol{G\cdot\wt{x}}\cap\wt{\mu}_\lie{p}^{-1}(0)\neq\emptyset\}$. It
is open in $\wt{X}$ (\cite{HeSt1}) and contains $(x_0,\gamma)$.

By analyticity of the action, the union $V$ of the open $G$-orbits in
$\mathcal{S}_G(\wt{\mu}_\lie{p}^{-1}(0))$ is dense in
$\mathcal{S}_G(\wt{\mu}_\lie{p}^{-1}(0))$. We note also that the union of the
open $G$--orbits is locally finite in $\mathcal{S}_G(\wt{\mu}_\lie{p}^{-1}(0))$
which can be seen as follows. For every $p\in\wt{\mu}_\lie{p}^{-1}(0)$ there
exists a slice neighborhood $G\cdot S\cong G\times_{G_x}S$ where $G_x$ is a
compatible subgroup of $G$ and $S$ can be viewed as an open neighborhood of $0$
in a $G_x$--representation space. Since $G_x$ has at most finitely many open
orbits in this representation space, we conclude that only finitely many open
$G$--orbits intersect the open set $G\cdot S$ which shows that the union of the
open $G$--orbits in $\mathcal{S}_G(\wt{\mu}_\lie{p}^{-1}(0))$ is locally
finite.

Let $W$ be the union of open $G$-orbits which contain $(x_0,\gamma)$ in their
closure and let $\ol{W}$ be the closure of $W$ in
$\mathcal{S}_G(\wt\mu_\lie{p}^{-1}(0))$. Then $W$ consists of only finitely
many open $G$--orbits and consequently $\ol{W}$ contains an open neighborhood
of $(x_0,\gamma)$. By Corollary~11.18 in \cite{HeSchw}, $\ol{W}$ intersects
$\wt{\mu}_\lie{p}^{-1}(0)$ in $K\cdot(x_0,\gamma)$. Therefore
$K\cdot(x_0,\gamma)$ is isolated in $\wt{\mu}_\lie{p}^{-1}(0)$ which shows that
the quotient $\wt{\mu_\lie{p}}^{-1}(0)/K$ is discrete. Then
$\mu_\lie{p}^{-1}(\gamma)/M$ is discrete by Lemma~\ref{lemma:FaserQuotienten}
which means that the $M$--orbits in $\mu_\lie{p}^{-1}(\gamma)$ are open. But
$M<K^\gamma$ so the $K^\gamma$--orbits are open in $\mu_\lie{p}^{-1}(\gamma)$
as well.
\end{proof}

This completes the proof of Theorem~\ref{Thm:Main}.

\begin{cor}
Let $X$ be a spherical $G$--gradient manifold. Then every $G$--stable
real-analytic submanifold $Y$ of $X$ is also spherical.
\end{cor}

\begin{proof}
The claim follows from the facts that $Y$ is a $G$--gradient manifold with
respect to $\mu_\lie{p}|_Y$ and that $\mu_\lie{p}|_Y$ almost separates the
$K$--orbits in $Y$ since this is true for $\mu_\lie{p}$.
\end{proof}

\begin{cor}
If one $G$--gradient map locally almost separates the $K$--orbits in $X$, then
every $G$--gradient map on $X$ almost separates the $K$--orbits.
\end{cor}

\section{Applications}

\subsection{Homogeneous semi-stable spherical gradient manifolds}

Let $G=K\exp(\lie{p})$ be connected real-reductive and let $X$ be a spherical
$G$--gradient manifold with gradient map $\mu_\lie{p}\colon X\to\lie{p}$. We
have seen in Lemma~\ref{Lem:OpenGOrbits} that $G$ has an open orbit in $X$. In
this subsection we consider the case that $X=G/H$ is homogeneous. In addition,
we suppose that $X$ is semi-stable, i.\,e.\ that $X=\mathcal{S}_G(
\mathcal{M}_\lie{p})$ holds. Consequently, we may assume that $H$ is of the
form $H=K_H\exp(\lie{p}_H)$ with $K_H=K\cap H$ and
$\lie{p}_H=\lie{p}\cap\lie{h}$.

\begin{rem}
The class of homogeneous semi-stable spherical gradient manifolds generalizes
the class of homogeneous affine spherical varieties in the complex setting.
\end{rem}

Let $\lie{p}=\lie{p}_H\oplus\lie{p}_H^\perp$ be a $K_H$--invariant
decomposition; then we have the Mostow decomposition $G/H\cong K\times_{K_H}
\lie{p}_H^\perp$ (see Theorem~9.3 in~\cite{HeSchw} for a proof which uses
gradient maps). Since $X$ is spherical, we conclude from Theorem~\ref{Thm:Main}
that the Mostow gradient map $\mu_\lie{p}\colon G/H\cong
K\times_{K_H}\lie{p}_H^\perp\to\lie{p}$, $[k,\xi]\mapsto\Ad(k)\xi$, almost
separates the $K$--orbits. In other words, the inclusion
$\lie{p}_H^\perp\hookrightarrow\lie{p}$ induces a map
$\lie{p}_H^\perp/K_H\to\lie{p}/K$ which has discrete fibers. This discussion
proves the following

\begin{prop}
Let $X=G/H$ be a semi-stable homogeneous $G$--gradient manifold and suppose
that $H=K_H\exp(\lie{p}_H)$ is compatible in $G=K\exp(\lie{p})$. Then $X$ is
spherical if and only if the map $\lie{p}_H^\perp/K_H\to\lie{p}/K$ induced by
the inclusion $\lie{p}_H^\perp\hookrightarrow\lie{p}$ has discrete fibers.
\end{prop}

\begin{ex}
For $H=\{e\}$ we have $K_H=\{e\}$ and $\lie{p}_H^\perp=\lie{p}$. Consequently,
$X=G$ is spherical if and only if the quotient map $\lie{p}\to\lie{p}/K$ has
discrete fibers, i.\,e.\ if and only if $K$ acts trivially on $\lie{p}$.
\end{ex}

Finally, we show that reductive symmetric spaces are spherical. Recall that
$G/H$ is a reductive symmetric space if there is an involutive automorphism
$\tau$ on $G$ such that $(G^\tau)^0\subset H\subset G^\tau$ holds. In this
situation we may assume without loss of generality that $\tau$ commutes with
the Cartan involution $\theta$. Hence, $H=K^\tau\exp(\lie{p}^\tau)$ is
compatible. In order to show that $X=G/H$ is spherical, we must prove that
$\lie{p}^{-\tau}/K^\tau\to\lie{p}/K$ has discrete fibers. From
$[\lie{p}^{-\tau},\lie{p}^{-\tau}]\subset\lie{k}^\tau$ we conclude that every
$K^\tau$--orbit in $\lie{p}^{-\tau}$ intersects a maximal Abelian subspace
$\lie{a}_0\subset\lie{p}^{-\tau}$ in an orbit of the finite group
$W_0:=\mathcal{N}_{K^\tau}(\lie{a}_0)/\mathcal{Z}_{K^\tau}(\lie{a}_0)$.
Extending $\lie{a}_0$ to a maximal Abelian subspace $\lie{a}$ of $\lie{p}$ we
see that $\lie{p}^{-\tau}/K^\tau\cong\lie{a}_0/W_0\to\lie{a}/W\cong\lie{p}/K$
has indeed finite fibers. Therefore we have proven the following

\begin{prop}
Let $X=G/H$ be a semi-stable homogeneous gradient manifold. If $H$ is a
symmetric subgroup of $G$, then the Mostow gradient map $\mu_\lie{p}\colon X
\to\lie{p}$ has finite fibers.
\end{prop}

\subsection{Relation to multiplicity-free representations}

Let $X$ be a real-analytic $G$--gradient manifold. Then $G$ acts linearly on
the space ${\mathcal{C}}^\omega(X)$ of complex-valued real-analytic functions
on $X$. Since $G$ is a compatible subgroup of some complex-reductive group
$U^\mbb{C}$, we observe that $G$ embeds as a closed subgroup into its
complexification $G^\mbb{C}$. Moreover, if $G$ contains no non-compact Abelian
factors, then $G^\mbb{C}$ is complex-reductive.

\begin{prop}\label{Prop:Multfree}
Suppose that $G$ acts properly on $X$ and that $G^\mbb{C}$ is
complex-reductive. If the $G$--representation on ${\mathcal{C}}^\omega(X)$ is
multiplicity-free, then $X$ is spherical.
\end{prop}

\begin{proof}
As is proven in~\cite{He}, there exists a Stein $G^\mbb{C}$--manifold
$X^\mbb{C}$ such that $X$ admits a $G$--equivariant embedding as a closed
maximally totally real submanifold into $X^\mbb{C}$. According to the example
discussed in Section~\ref{Section:Ex} it suffices to show that $X^\mbb{C}$ is
$G^\mbb{C}$--spherical.

In order to see this, note that the restriction mapping $\mathcal{O}(X^\mbb{C})
\to\mathcal{C}^\omega(X)$ is injective and $G$--equivariant. This implies that
the $G$-- (and hence also the $G^\mbb{C}$--)representation on
$\mathcal{O}(X^\mbb{C})$ is multiplicity-free. Therefore, Theorem~2
in~\cite{AkHe} applies to show that $X^\mbb{C}$ is spherical which finishes the
proof.
\end{proof}

\begin{rem}
In Proposition~\ref{Prop:Multfree} properness of the $G$--action on $X$ is
needed to guarantee the existence of the complexification $X^\mbb{C}$. If
$X=G/H$ is homogeneous, then we may take $X^\mbb{C}:=G^\mbb{C}/H^\mbb{C}$ and
the same argument as above shows: If the $G$--representation on
$\mathcal{C}^\omega(G/H)$ is multiplicity-free, then $G/H$ is spherical.
\end{rem}

Even if we assume that $G$ acts properly on $X$, the converse of
Proposition~\ref{Prop:Multfree} does not hold as the following example shows.

\begin{ex}
Let $G=K$ be a compact Lie group acting by left multiplication on $X=K$. Then
$\mu_\lie{p}\equiv0$ separates the $K$--orbits in $X$ but the
$K$--representation on $\mathcal{C}^\omega(K)$ is not multiplicity-free which
can be deduced from the Peter-Weyl Theorem.
\end{ex}

However, there is a special class of real-reductive Lie groups for which the
proof of the complex multiplicity-freeness result generalizes to the real
situation. A real-reductive Lie group $G$ belongs to this class if the minimal
parabolic subalgebras $\lie{q}_0=\lie{m}\oplus\lie{a}\oplus\lie{n}$ are
solvable, i.\,e.\ if $\lie{m}$ is Abelian.

\begin{ex}
Among the classical semi-simple Lie groups this is the case e.\,g.\ for
${\rm{SL}}(n,\mbb{R})$, ${\rm{Sp}}(n,\mbb{R})$, ${\rm{SU}}(p,p)$,
${\rm{SO}}(p,p)$ and ${\rm{SO}}(p,p+1)$ (see Appendix~C.3 in~\cite{Kn}).
\end{ex}

\begin{lem}
Let $X$ be a spherical $G$--gradient manifold. If the minimal parabolic
subalgebras of $\lie{g}$ are solvable, then the $G$--representation on
$\mathcal{C}^\omega(X)$ is multiplicity-free.
\end{lem}

\begin{proof}
We must show that $\dim\Hom_G\bigl(V,\mathcal{C}^\omega(X)\bigr)\leq1$ holds
for every complex finite-dimensional irreducible $G$--module $V$. Let $Q_0=MAN$
be a minimal parabolic subgroup of $G$ and let $V$ be a complex
finite-dimensional irreducible $G$--module. By Engel's Theorem the space $V^N$
of $N$--invariant vectors has positive dimension. The restriction map induces a
linear map
\begin{equation*}
\Hom_G\bigl(V,\mathcal{C}^\omega(X)\bigr)\to\Hom_{MA}\bigl(V^N,
\mathcal{C}^\omega(X)^N\bigr),
\end{equation*}
which is injective since $V^N$ generates $V$ as a $G$--module. Hence, it is
enough to show $\dim\Hom_{MA}\bigl(V^N,\mathcal{C}^\omega(X)^N\bigr)\leq1$. Let
us assume the contrary. Then there are linearly independent functions $f_1,f_2
\in\mathcal{C}^\omega(X)^N$ which transform under the same character of the
Abelian group $M^0A$. Consequently, the quotient $f_1/f_2$ is a real-analytic
function defined on the dense open set $\{f_2\not=0\}$ and invariant under
$Q_0^0=M^0AN$. Since this contradicts the assumption that $Q_0$ has an open
orbit in $X$, the proof is finished.
\end{proof}

\subsection{Open Borel-orbits are Stein}

In this subsection we consider the holomorphic situation, i.\,e.\ $G=U^\mbb{C}$
is complex-reductive and acts holomorphically on the K\"ahler manifold $Z$ such
that the $U$--action is Hamiltonian with moment map $\mu\colon Z\to\lie{u}^*$.
In Section~\ref{section:proof} we have given a new proof of the following
result of Brion.

\begin{thm}
The moment map $\mu\colon Z\to\lie{u}^*$ separates the $U$--orbits in $Z$ if
and only if $Z$ is spherical, i.\,e.\ if a Borel subgroup $B\subset G$ has an
open orbit in $Z$.
\end{thm}

In this subsection we will show that our proof further implies that the open
$B$--orbit in $Z$ is Stein.

\begin{prop}
If the moment map $\mu\colon Z\to\lie{u}^*$ separates the $U$--orbits in $Z$,
then the open $B$--orbit in $Z$ is Stein.
\end{prop}

\begin{proof}
Let $z\in Z$ be a generic element and let $Q\subset G$ be the parabolic
subgroup associated to $\mu(z)$. Consequently, the zero fiber of the shifted
moment map on the K\"ahler manifold $Z\times(G/Q)$ is non-empty. We may assume
without loss of generality that the element $(z,eQ)\in Z\times(G/Q)$ is
contained in this zero fiber. By Proposition~\ref{Prop:ProofSpecialCase} the
orbit $G\cdot(z,eQ)$ is open in $Z\times(G/Q)$ which in turn implies that
$Q\cdot z$ is open in $Z$. Moreover, since $(z,eQ)$ lies in the zero fiber of a
moment map, the isotropy $G_{(z,eQ)}=G_z\cap Q=Q_z$ is complex-reductive which
proves that $Q\cdot z\cong Q/Q_z$ is Stein (see Theorem~5 in~\cite{MatMo}). The
open $B$--orbit in $Z$ must be contained in $Q\cdot z$ and is therefore
holomorphically separable. Applying a result of Huckleberry and Oeljeklaus
(\cite{HuOel}) we finally see that the open $B$--orbit is Stein.
\end{proof}

\end{document}